\documentclass{article}
\usepackage[utf8]{inputenc}
\usepackage{fullpage}
\usepackage{amsthm,mathtools,scalerel}
\usepackage{color,soul,latexsym,amsmath,amssymb,amsfonts,amsthm,dsfont,enumitem,xcolor,bbm,threeparttable,graphicx,float,subfigure}
\usepackage{authblk}
\usepackage[round]{natbib}
\RequirePackage[colorlinks=true,allcolors=blue,hypertexnames=false]{hyperref}%
\usepackage{graphicx,todonotes}
\usepackage{outlines}
\usepackage{array}
\usepackage{selectp}

\usepackage[ruled,linesnumbered, vlined, noend]{algorithm2e}
\mathtoolsset{showonlyrefs}
\usepackage{titlesec}

\newtheorem{thm}{Theorem}

\newtheorem{lemma}{Lemma}

\newtheorem{definition}{Definition}
\newtheorem*{remark}{Remark}

\title{\bf Median Regularity and Honest Inference}
\author{Arun Kumar Kuchibhotla}
\author{Sivaraman Balakrishnan}
\author{Larry Wasserman}
\affil{\texttt{\{arunku,\,siva,\,larry\}@stat.cmu.edu}}
\affil{Department of Statistics \& Data Science, Carnegie Mellon University}

\date{}                   
\begin{document}

\maketitle

\begin{abstract}
We introduce a new notion of regularity of an estimator called median regularity. We prove that uniformly valid (honest) inference for a functional is possible if and only if there exists a median regular estimator of that functional. To our knowledge, such a notion of regularity that is \emph{necessary} for uniformly valid inference is unavailable in the literature.
\end{abstract}
\maketitle

\section{Introduction}
There is a long standing distinction between pointwise and uniform inference for functionals in statistics. Pointwise confidence intervals have asymptotically correct coverage as the distribution of the data is fixed as the sample size diverges. On the other hand, uniform confidence intervals have asymptotically correct coverage even as the distribution of the data changes (i.e., triangular array) with the sample size. It is well-known that confidence intervals that only satisfy pointwise coverage validity can exhibit poor finite-sample behavior~\citep{potscher2002lower}. This phenomenon is very well explained through the famous Hodges estimator in the normal mean example; more complicated examples can be found in post-model selection problems~\citep{leeb2005model}. The distinction between pointwise and uniform inference has also been discussed in the context of resampling techniques such as bootstrap and subsampling; see~\cite{romano2012uniform,andrews2010asymptotic}. Traditionally, uniformly valid confidence intervals are constructed based on a \emph{regular} estimator~\citep{van1991differentiable,pfanzagl2000regularly,hirano2012impossibility}.

Suppose we have i.i.d. data $X_1, \ldots, X_n$ from a regular parametric model $p_{\theta_0}\in\{p_{\theta}:\,\theta\in\Theta\}$ for some open set $\Theta$. An estimator $\widehat{\tau}_n$ of $\tau = g(\theta_0)\in\mathbb{R}$ is said to be \emph{regular} if for some $r_n\to\infty$ and for any $h$, under $p_{\theta_0 + h/r_n}$,
\begin{equation}\label{eq:classical-regularity}
r_n\left(\widehat{\tau}_n - g(\theta_0 + h/\sqrt{n})\right) ~\overset{d}{\to}~ L_{\theta_0}, 
\end{equation}
for some distribution $L_{\theta_0}$ independent of $h$ as $n\to\infty$. Regularity in the sense of~\eqref{eq:classical-regularity} is a form of continuity of the limiting distribution (in $\theta_0$). Informally,~\eqref{eq:classical-regularity} says that the limiting distribution does not depend on the perturbation direction $h$. The theory of parametric and semiparametric efficiency is built around this notion of regularity defined in~\eqref{eq:classical-regularity}. Under the classical setting of parametric models that are locally asymptotically normal (LAN), the maximum likelihood estimator can be shown to be a regular estimator. 

Traditionally, inference based on a regular estimator follows by estimating $L_{\theta_0}$ (or the unknown quantities in $L_{\theta_0}$) based on the observed data. Because $L_{\theta_0}$ is the \emph{uniform} limiting distribution of the estimator in the neighborhood, the coverage of the confidence intervals thus obtained are locally uniformly valid. The regularity notion~\eqref{eq:classical-regularity} is currently the only general method of deriving (locally) uniformly valid inference~\citep{hirano2012impossibility}. However, a less known fact is that one does not need an estimator $\widehat{\tau}_n$ to be regular in the sense of~\eqref{eq:classical-regularity} in order to achieve uniformly valid inference, i.e., the estimator need not be regular in order to construct a uniformly valid confidence interval based on it. By validity here we mean that the miscoverage probability at level $\alpha$ is bounded above by $\alpha$. A simple example is the estimator $\widehat{\tau}_n = \overline{X}_n{1}\{\overline{X}_n \ge 0\}$ for the functional $\tau = \mu{1}\{\mu \ge 0\}$. The non-regular behavior of this estimator is well-known; see, for example,~\citet[Appendix 1]{robins2004optimal}. On the other hand, there are several ways to construct uniformly valid intervals based on the estimator $\widehat{\tau}_n$ (for instance, via the HulC procedure of~\citet{kuchibhotla2021hulc}).

This raises some important questions: is there a notion of regularity of the problem/estimator needed for uniformly valid inference? If so, what is that notion of regularity? We define a new notion of regularity called \emph{median regularity} of an estimator and show that pointwise or uniformly valid inference is possible if and only if there exists an estimator which is pointwise or uniformly median regular, respectively. Furthermore, it will be clear that an estimator that is uniformly median regular leads to uniformly valid inference. We will only focus on independent and identically distributed data throughout.
\section{Definitions and Main Result}
Suppose $X_1, \ldots, X_n$ are independent and identically distributed random variables with distribution $P$. Let $\theta = \theta(P)$ be a scalar functional of interest. For any estimator $\widehat{\theta}_n$ based on $X_1, \ldots, X_n$, define the median bias as
\begin{equation}\label{eq:median-bias}
\mbox{Med-bias}_{\theta}(\widehat{\theta}_n) = \left(\frac{1}{2} - \min\left\{\mathbb{P}(\widehat{\theta}_n \ge \theta_0),\,\mathbb{P}(\widehat{\theta}_n \le \theta_0)\right\}\right)_+,
\end{equation}
where $(x)_+ = \max\{x, 0\}$. We will write $\mbox{Med-bias}_{\theta(P)}(\widehat{\theta}_n; P)$ when we want to stress that $\theta = \theta(P)$ and the probabilities in the definition are computed when the underlying data are generated from $P$. 

An estimator $\widehat{\theta}_n$ is said to be median unbiased if $\mbox{Med-bias}_{\theta(P)}(\widehat{\theta}_n) = 0$~\citep{pfanzagl2011parametric}. This is a finite sample property as defined and is satisfied if the probability of over-estimation and under-estimation are both at least $1/2$, i.e., 
\[
\mathbb{P}(\widehat{\theta}_n \ge \theta(P)) \ge \frac{1}{2}\quad\mbox{and}\quad \mathbb{P}(\widehat{\theta}_n \le \theta(P)) \ge \frac{1}{2}.
\]
There exist several settings in the literature where finite sample median unbiased estimators exist. We only mention~\cite{pfanzagl1979optimal} that presents the most general result we know for construction of finite sample median unbiased estimators and refer the reader to Section 3 of~\cite{kuchibhotla2021hulc} for further references. Here, we present here four simple examples for illustration:
\begin{enumerate}
    \item Suppose $X_1, \ldots, X_n$ are i.i.d. random variables from $N(\theta_0, 1)$. Then the sample mean $\overline{X}_n$ is median unbiased for $\theta_0$. The same holds true for any symmetric location family.
    \item Suppose $X_1, \ldots, X_n$ are i.i.d. real valued random variables from some distribution $P$. Let the functional of interest be $\theta(P)$, the median of $P$. Then for any $r \le n/2$ the estimator
    \[
    \widehat{\theta}_n = \begin{cases}X_{(r)},&\mbox{with probability } 1/2,\\X_{(n-r+1)}, &\mbox{with probability }1/2,\end{cases}
    \]
    is median unbiased for $\theta(P)$. This holds irrespective of whether the underlying distribution $P$ has a density or not, and irrespective of whether this density is bounded away from 0 at the median or not; see~\citet[Section 4]{mahamunulu1969estimation}.
    \item Suppose $X_1, \ldots, X_n$ are i.i.d. random variables from $\mbox{Unif}(0, \theta_0)$. Then $\widehat{\theta}_n = X_{(n)} - 2X_{(n-1)}$ is median unbiased for $\theta_0$~\citep{loh1984estimating}. The MLE $X_{(n)}$ has a median bias of $1/2$, the largest possible median bias for any estimator.
    \item Suppose $X_1, \ldots, X_n$ are i.i.d. random variables from $N(\mu, 1)$. Let the functional of interest be $\kappa(\mu) = \mu\mathbbm{1}\{\mu \ge 0\}$. Then the estimator $\kappa(\overline{X}_n)$ is median unbiased for $\kappa(\mu)$~\citep[Section 3.7]{kuchibhotla2021hulc}.
\end{enumerate}
Although simple median unbiased estimators exist in some parametric models, the situation is complicated in semi-/non-parametric models where one typically relies on asymptotics. We define asymptotic notions of median unbiasedness below for a family of distributions $\mathcal{P}$.
Before describing these asymptotic notions, we remark that one can always define an estimator that takes two values $-\infty$ and $\infty$ with probability $1/2$ as a median unbiased estimator of any real-valued functional $\theta(P)$. This is a trivial and useless estimator for practical purposes. In the following, we will ignore such a trivial estimator and call an estimator \emph{non-trivial} if $\mathbb{P}(|\widehat{\theta}_n| < \infty) > 0$. 
\begin{definition}
A sequence of estimators $\{\widehat{\theta}_n:\,n\ge1\}$ of $\theta(P)$ is said to be \textbf{(asymptotically) median unbiased for a set of distributions $\mathcal{P}$} if
\[
\sup_{P\in\mathcal{P}}\limsup_{n\to\infty}\,\mathrm{Med}\mbox{-}\mathrm{bias}_{\theta(P)}(\widehat{\theta}_n; P) = 0.
\]
\end{definition}
Asymptotic median unbiasedness is a pointwise notion of median unbiasedness in the sense that we only require the median bias to converge to zero for every $P\in\mathcal{P}$.
\begin{definition}
A sequence of estimators $\{\widehat{\theta}_n:\,n\ge1\}$ of $\theta(P)$ is said to be \textbf{(asymptotically) median regular for a set of distributions $\mathcal{P}$} if
\[
\limsup_{n\to\infty}\,\sup_{P\in\mathcal{P}}\,\mathrm{Med}\mbox{-}\mathrm{bias}_{\theta(P)}(\widehat{\theta}_n; P) = 0.
\]
\end{definition}
Let us define a confidence interval procedure as a collection of intervals $\{\widehat{\mathrm{CI}}_{\alpha,n}:\alpha\in[0,1], n\ge1\}$ with $\widehat{\mathrm{CI}}_{\alpha,n}$ computable based on $n$ observations. This means that a confidence interval procedure is just a function of the data. Below we define different notions of validity of confidence intervals. Similar to non-trivial estimators, we ignore trivial confidence intervals which return $[-\infty, \infty]$ and call a confidence interval \emph{non-trivial} if at least one of the end points of the confidence interval is finite almost surely.
\begin{definition}
A confidence interval procedure $\{\widehat{\mathrm{CI}}_{\alpha,n}:\,\alpha\in[0,1],n \ge1\}$ for a functional $\theta(P)$ is said to be \textbf{(asymptotically) pointwise valid at level $\gamma$ for a set of distributions $\mathcal{P}$} if
\[
\sup_{P\in\mathcal{P}}\limsup_{n\to\infty}\,\mathbb{P}_P(\theta(P)\notin\widehat{\mathrm{CI}}_{\gamma,n}) \le \gamma.
\]
\end{definition}
\begin{definition}
A confidence interval procedure $\{\widehat{\mathrm{CI}}_{\alpha,n}:\,\alpha\in[0,1],n \ge1\}$ for a functional $\theta(P)$ is said to be \textbf{(asymptotically) uniformly valid at level $\gamma$ for a set of distributions $\mathcal{P}$} if
\[
\limsup_{n\to\infty}\,\sup_{P\in\mathcal{P}}\mathbb{P}_P(\theta(P)\notin\widehat{\mathrm{CI}}_{\gamma,n}) \le \gamma.
\]
\end{definition}
Asymptotically uniformly valid confidence intervals are also sometimes referred to as honest intervals~\citep{li1989honest,potscher2002lower}.
\begin{definition}
A confidence interval procedure $\{\widehat{\mathrm{CI}}_{\alpha,n}:\,\alpha\in[0,1],n \ge1\}$ for a functional $\theta(P)$ is said to be \textbf{(asymptotically) uniformly valid for a set of distributions $\mathcal{P}$} if there exists $\tau_n = \tau_n(\mathcal{P})\to0$ such that
\[
\limsup_{n\to\infty}\sup_{\gamma\in[\tau_n,1]}\sup_{P\in\mathcal{P}}\left(\mathbb{P}_P(\theta(P)\notin\widehat{\mathrm{CI}}_{\gamma,n}) - \gamma\right)_+ = 0.
\]
\end{definition}
The ordering of the supremums over $\gamma\in[\tau_n, 1]$ and $P\in\mathcal{P}$ is to make clear that $\tau_n$ is independent of $P$, it is only allowed to be a function of $\mathcal{P}$ and the sample size

The following is the main result of this paper proved in Section~\ref{sec:proof-of-equivalence-median-regular-uniform-inference}.
\begin{thm}\label{thm:equivalence-median-regular-uniform-inference}
For any set $\mathcal{P}$ of distributions (independent of the sample size $n$), the following three statements are equivalent for independent and identically distributed data:
\begin{enumerate}[label=(\arabic*)]
    \item There exists an (asymptotically) uniformly valid non-trivial confidence interval procedure at level $\gamma\in(0, 1)$ for $\theta(P)$ for the set of distributions $\mathcal{P}$.\label{eq:valid-at-gamma}
    \item There exists an (asymptotically) uniformly valid non-trivial confidence interval procedure for $\theta(P)$ for the set of distributions $\mathcal{P}$.\label{eq:valid-always}
    \item There exists a non-trivial (asymptotically) median regular estimator sequence for $\theta(P)$ for the set of distributions $\mathcal{P}$. \label{eq:median-regular}
\end{enumerate}
\end{thm}

Theorem~\ref{thm:equivalence-median-regular-uniform-inference} although is written for a set of distributions $\mathcal{P}$, there is no requirement on its cardinality; it can be an infinite set or a singleton. This flexibility allows us to use Theorem~\ref{thm:equivalence-median-regular-uniform-inference} to derive simple corollaries. For example, take $\mathcal{P} = \{P_0\}$ to be a singleton with only one distribution that does not change with the sample size. Then Theorem~\ref{thm:equivalence-median-regular-uniform-inference} implies that asymptotically \emph{pointwise} confidence intervals exist if and only if asymptotically median unbiased estimators exist. 
\paragraph{Finite sample version and impossibility results.}
It is worth noting that a finite sample version of Theorem~\ref{thm:equivalence-median-regular-uniform-inference} can be obtained from its proof. Furthermore, one can prove that there exists a finite sample valid confidence interval at some level $\gamma\in(0, 1)$ for $\theta(P)$ if and only if for any $\delta\in(0, 1/2)$, there exists an estimator for $\theta(P)$ with median bias bounded by $\delta$. This can be seen from Lemmas~\ref{lem:confidence-median-bias} and~\ref{lem:median-bias-confidence} with $r_n = s_n = 0$ for all $n\ge1$.

With this finite sample version, Theorem~\ref{thm:equivalence-median-regular-uniform-inference} can also be used to derive some impossibility results. For example, there exist many functionals (in non-parametric models) for which no finite sample non-trivial confidence intervals exist; the population mean is an example of such a functional~\citep{bahadur1956nonexistence,pfanzagl1998nonexistence}. In such cases, the finite sample version of Theorem~\ref{thm:equivalence-median-regular-uniform-inference} implies that no estimator that is finite almost surely can control the median bias for such functionals; the estimator must take $\infty$ or $-\infty$ with non-zero probability to control the median bias at a value less than $1/2$. Similarly, there exist functionals for which only finite sample one-sided up/down confidence intervals exists; the number of modes of a density is an example~\citep{donoho1988one}. In such cases again, no estimator that is finite almost surely can have median bias bounded away from $1/2$. 

\paragraph{Rate of convergence and optimality.} Unlike the classical notion of regularity~\eqref{eq:classical-regularity} which explicitly indicates the rate of convergence of the estimator $\widehat{\theta}_n$, the median regularity of an estimator does not mention/need any rate of convergence of the estimator. In fact, the estimator can be inconsistent for the true functional. This lack of reference to rate of convergence is in fact necessary, i.e., if a notion of regularity is equivalent to valid inference for all functionals, then there should be no requirement on the existence of a consistent estimator for the functionals. For example, there are functionals which are only partially identified and for such functionals no consistent estimator exists. There do exist valid confidence intervals for these functionals~\citep{imbens2004confidence,Fan2010}. Such absence of reference to rate of convergence or consistency is made possible only because we are concerned with valid coverage and not about the width of the confidence intervals. A discussion of the width of the confidence interval will typically make reference to the consistency, rate of convergence, and ``optimality'' of the estimator. We note in passing that optimal median regular estimators have been considered by~\cite{pfanzagl1970asymptotic} and~\cite{michel1974bounds}.
\section{Proof of Theorem~\ref{thm:equivalence-median-regular-uniform-inference}}\label{sec:proof-of-equivalence-median-regular-uniform-inference}
\subsection{Intermediate Lemmas}
The following key lemma shows that the existence of a valid confidence interval implies the existence of an estimator with certain bounds on its median bias. 
\begin{lemma}\label{lem:confidence-median-bias}
Fix any distribution $P$. Fix a known $\gamma\in(0, 1)$. Suppose for any sample of $n$ independent and identically distributed observations from $P$, there exists a confidence interval $\widetilde{\mathrm{CI}}_{\gamma,n}$ and $r_n\in[0,1)$ satisfying
\[
\mathbb{P}\left(\theta(P)\notin \widetilde{\mathrm{CI}}_{\gamma,n}\right) ~\le~ \gamma + r_n,
\]
then 
\begin{enumerate}
    \item there exists an estimator $\widehat{\theta}_n$ such that
    \[
    \mathrm{Med}\mbox{-}\mathrm{bias}_{\theta(P)}(\widehat{\theta}_n; P) ~\le~ \frac{\gamma + r_n}{2},
    \]
    \item for any $\alpha\in[0, 1]$, there exists a confidence interval $\widehat{\mathrm{CI}}_{\alpha, n}$ such that
    \[
    \mathbb{P}\left(\theta(P) \notin \widehat{\mathrm{CI}}_{\alpha, n}\right) ~\le~ \alpha\left(1 + 2r_{\lfloor n/B_{\alpha,\gamma/2}\rfloor}\right)^{B_{\alpha,\gamma/2}},
    \]
    where $B_{\alpha,\gamma/2}$ is the smallest integer $B$ such that $(1-\gamma)^B + (1 + \gamma)^B \le 2^B\alpha$.
\end{enumerate}
\end{lemma}
\begin{proof}
Let $\widetilde{\mathrm{CI}}_{\gamma,n} = [\ell_n, u_n]$.
Consider the estimator
\[
\widehat{\theta}_n = 
\begin{cases}
\ell_n, &\mbox{with probability $1/2$},\\
u_n, &\mbox{with probability $1/2$.}
\end{cases}
\]
Observe that
\begin{align*}
\mathbb{P}(\widehat{\theta}_n \le \theta(P)) &= \mathbb{P}(\ell_n \le \theta(P))\times\frac{1}{2} + \mathbb{P}(u_n \le \theta(P))\times\frac{1}{2}\\
&\ge \mathbb{P}(\ell_n \le \theta(P) \le u_n)\times\frac{1}{2} + 0\times\frac{1}{2}\\
&\ge (1 - \gamma - r_n)_+/2.
\end{align*}
Similarly,
\begin{align*}
\mathbb{P}(\widehat{\theta}_n \ge \theta(P)) &= \mathbb{P}(\ell_n \ge \theta(P))\times\frac{1}{2} + \mathbb{P}(u_n \ge \theta(P))\times\frac{1}{2}\\
&\ge 0\times\frac{1}{2} + \mathbb{P}(\ell_n \le \theta(P) \le u_n)\times\frac{1}{2}\\
&\ge (1 - \gamma - r_n)_+/2.
\end{align*}
Therefore,
\[
\min\left\{\mathbb{P}(\widehat{\theta}_n \le \theta(P)), \mathbb{P}(\widehat{\theta}_n \ge \theta(P))\right\} \ge \frac{(1 - \gamma - r_n)_+}{2},
\]
and hence,
\[
\mbox{Med-bias}_{\theta(P)}(\widehat{\theta}_n; P) \le \frac{\gamma + r_n}{2}.
\]
This proves that the existence of a confidence interval implies the existence of an estimator with a specific bound on the median bias.

To construct a confidence interval at arbitrary level $\alpha$, Theorem 2 (and Remark 2.2) of~\cite{kuchibhotla2021hulc} using the estimator $\widehat{\theta}_n$, its finite sample bias as approximately $\gamma/2$, and $B = B_{\alpha,\gamma/2}$ yields a confidence interval $\widehat{\mathrm{CI}}_{\alpha,n}$ satisfying
\[
\mathbb{P}(\theta(P)\notin \widehat{\mathrm{CI}}_{\alpha,n}) \le \alpha\left(1 + 2r_{\lfloor n/B_{\alpha,\gamma/2}\rfloor}\right)^{B_{\alpha,\gamma/2}}.
\]
This completes the proof.
\end{proof}
\begin{remark}
One can construct an alternative $\widehat{\mathrm{CI}}_{\alpha,n}$ by considering union of independent confidence intervals. If we split the data into $B = \lceil\log_{\gamma}\alpha\rceil$ batches each of size $\lfloor{n/B}\rfloor$ and construct confidence intervals $\widetilde{\mathrm{CI}}^{(j)}_{\gamma,\lfloor{n/B}\rfloor}$ on each batch. 
Consider the confidence interval
\[
\widehat{\mathrm{CI}}_{\alpha, n} = \bigcup_{j=1}^{B} \widetilde{\mathrm{CI}}_{\gamma, \lfloor{n/B}\rfloor}^{(j)}.
\]
Observe that
\begin{align*}
\mathbb{P}\left(\theta(P) \notin \widehat{\mathrm{CI}}_{\alpha,n}\right) &= \prod_{j=1}^B \mathbb{P}(\theta(P)\notin\widetilde{\mathrm{CI}}_{\gamma,\lfloor{n/B}\rfloor}) \le (\gamma + r_{\lfloor{n/B}\rfloor})^B \leq \alpha(1 + r_{\lfloor{n/B}\rfloor}/\gamma)^B.
\end{align*}
If $\alpha \ge \gamma$, then take $B = 1$ in the above calculation; this just returns the confidence interval $\widetilde{\mathrm{CI}}_{\gamma, n}$. Hence, we get $\widehat{\mathrm{CI}}_{\alpha,n}$ that satisfies
\[
\mathbb{P}(\theta(P)\notin\widehat{\mathrm{CI}}_{\alpha,n}) \le \begin{cases}\alpha \left(1 + \frac{r_{\lfloor n/\lceil\log_{\gamma}\alpha\rceil\rfloor}}{\gamma}\right)^{\lceil\log_{\gamma}\alpha\rceil},&\mbox{if }\alpha \le \gamma,\\
\alpha, &\mbox{if }\alpha > \gamma.\end{cases}
\]
\end{remark}
\begin{lemma}\label{lem:median-bias-confidence}
Fix any distribution $P$. Fix a known $\gamma\in(0, 1)$. Suppose for any sample of $n$ independent and identically distributed observations from $P$, there exists an estimator $\widehat{\theta}_n$ and $s_n \in [0, 1)$ satisfying
\[
\mathrm{Med}\mbox{-}\mathrm{bias}_{\theta(P)}(\widehat{\theta}_n; P) \le \gamma + s_n,
\]
then for any $\alpha\in(0, 1)$, there exists a confidence interval $\widehat{\mathrm{CI}}_{\alpha,n}$ such that
\[
\mathbb{P}\left(\theta(P)\notin \widehat{\mathrm{CI}}_{\alpha,n}\right) ~\le~ \alpha\left(1 + 2s_{\lfloor n/B_{\alpha,\gamma}\rfloor}\right)^{B_{\alpha,\gamma}}.
\]
\end{lemma}
\begin{proof}
The proof is exactly the same as the proof of Lemma~\ref{lem:confidence-median-bias}(2).
\end{proof}
\subsection{Proof of Theorem~\ref{thm:equivalence-median-regular-uniform-inference}}
We will prove the result by proving~\ref{eq:valid-at-gamma} $\Rightarrow$~\ref{eq:valid-always},~\ref{eq:valid-always} $\Rightarrow$~\ref{eq:median-regular}, and~\ref{eq:median-regular} $\Rightarrow$~\ref{eq:valid-at-gamma}. The aspect of non-triviality of estimators and confidence intervals follows from the proof and is not explicitly mentioned.

\paragraph{\ref{eq:valid-at-gamma} $\Rightarrow$~\ref{eq:valid-always}:} Let $\{\widetilde{\mathrm{CI}}_{\gamma,n}, n\ge1\}$ be an asymptotically uniformly valid confidence interval procedure at level $\gamma$, i.e., 
\[
\limsup_{n\to\infty} \sup_{P\in\mathcal{P}} \left(\mathbb{P}_P(\theta(P)\notin\widetilde{\mathrm{CI}}_{\gamma,n}) - \gamma\right)_+ = 0.
\]
Define
\[
r_n := \sup_{P\in\mathcal{P}}\left(\mathbb{P}_P(\theta(P)\notin\widetilde{\mathrm{CI}}_{\gamma,n}) - \gamma\right)_+.
\]
The hypothesis implies that $r_n \to 0$. Note that $r_n$ is not known in practice. Lemma~\ref{lem:confidence-median-bias} provides a construction of a confidence interval procedure $\{\widehat{\mathrm{CI}}_{\alpha,n}:\alpha\in[0, 1], n\ge1\}$ based on $\{\widetilde{\mathrm{CI}}_{\gamma,n}:n\ge1\}$ that satisfies for all $P\in\mathcal{P}$ and all $\alpha\in[0, 1]$:
\[
\mathbb{P}_P\left(\theta(P)\notin \widehat{\mathrm{CI}}_{\alpha,n}\right) \le \alpha\left(1 + 2r_{\lfloor n/B_{\alpha,\gamma/2}\rfloor}\right)^{B_{\alpha,\gamma/2}} \le \alpha\times\exp\left(2B_{\alpha,\gamma/2}r_{\lfloor n/B_{\alpha,\gamma/2}\rfloor}\right).
\]
This implies that
\begin{equation}\label{eq:miscoverage-bound-for-all-alpha}
\sup_{P\in\mathcal{P}}\left(\mathbb{P}_P(\theta(P)\notin\widehat{\mathrm{CI}}_{\alpha,n}) - \alpha\right)_+ \le \exp\left(2B_{\alpha,\gamma/2}r_{\lfloor n/B_{\alpha,\gamma/2}\rfloor}\right) - 1.
\end{equation}
Note that the right hand side of~\eqref{eq:miscoverage-bound-for-all-alpha} only depends on $\mathcal{P}$ through $r_{\lfloor n/B_{\alpha,\gamma/2}\rfloor}$. In order to prove that $\{\widehat{\mathrm{CI}}_{\alpha,n}:\alpha\in[0,1],n\ge1\}$ is asymptotically uniformly valid for $\mathcal{P}$, it suffices to find an $\alpha_n \to 0$ such that the right hand side of~\eqref{eq:miscoverage-bound-for-all-alpha} converges to zero uniformly over all $\alpha\in[\alpha_n, 1]$. This is equivalent to proving that $B_{\alpha,\gamma/2}r_{\lfloor n/B_{\alpha,\gamma/2}\rfloor}$ converges to zero uniformly over all $\alpha\in[\alpha_n,1]$. For notational convenience, set $\delta = \gamma/2$. Recall that $B_{\alpha,\delta}$ is defined as
\[
B_{\alpha,\delta} = \min\left\{B\in\{1, 2, \ldots\}:\, (1/2 - \delta)^B + (1/2 + \delta)^B \le \alpha\right\}.
\]
It is easily verified that $B_{\alpha,\delta}$ is an increasing function of $\alpha$ and satisfies
\begin{equation}\label{eq:inequalities-B-alpha-delta}
\left\lfloor\frac{\log(2/\alpha)}{\log(2)}\right\rfloor \le B_{\alpha,\delta} \le \left\lceil\frac{\log(2/\alpha)}{\log(2/(1 + 2\delta))}\right\rceil \le \frac{2\log(2/\alpha)}{\log(2/(1 + 2\delta))}.
\end{equation}
Now, define the new sequence $\{r_k^*\}_{k\ge1}$ based on $\{r_k\}_{k\ge1}$ as $r_k^* = \sup_{m\ge k}r_m$. It is clear that $r_k^*$ is decreasing in $k$ and $r_k^* \to 0$ as $k\to\infty$ if $r_k\to0$ as $k\to\infty$.
Because $\gamma$ and hence $\delta$ are fixed numbers as $n\to\infty$,
\[
\sup_{\alpha\in[\alpha_n,1]}\,\log(2/\alpha)r_{\lfloor n/B_{\alpha,\delta}\rfloor}^* \to 0\mbox{ as $n\to\infty$} \quad\Rightarrow\quad \sup_{\alpha\in[\alpha_n, 1]}\,B_{\alpha,\delta}r_{\lfloor n/B_{\alpha,\delta}\rfloor} \to 0\mbox{ as $n\to\infty$}.
\]
Given the monotonicity properties of $\log(2/\alpha)$ (in $\alpha$) and of $r^*_k$ in $k$, it is easy to prove that
\[
\sup_{\alpha\in[\alpha_n,1]}\,\log(2/\alpha)r_{\lfloor n/B_{\alpha,\delta}\rfloor}^* = \log(2/\alpha_n)r^*_{\lfloor n/B_{\alpha_n,\delta}\rfloor}.
\]
Therefore, finding $\alpha_n\to0$ as $n\to\infty$ such that the right hand side of~\eqref{eq:miscoverage-bound-for-all-alpha} converges to zero uniformly over all $\alpha\in[\alpha_n, 1]$ is same as finding $\alpha_n\to0$ such that $\log(2/\alpha_n)r^*_{\lfloor n/B_{\alpha_n,\delta}\rfloor}$ converges to zero. Define
\[
\alpha_n ~:=~ \max\left\{\frac{1}{n},\, 2\exp\left(-(r^*_{\lfloor n/B_{1/n,\delta}}\rfloor)^{-1/2}\right)\right\}.
\]
From inequalities~\eqref{eq:inequalities-B-alpha-delta}, it follows that $\lfloor n/B_{1/n,\delta}\rfloor \to \infty$ as $n\to\infty$ and hence, $r^*_{\lfloor n/B_{1/n,\delta}\rfloor} \to 0$ as $n\to\infty$. This implies that the choice of $\alpha_n$ above converges to zero as $n\to\infty$ for any $\delta > 0$. Furthermore, 
\[
r^*_{\lfloor n/B_{\alpha_n,\delta}\rfloor}\log(2/\alpha_n) \le r^*_{\lfloor n/B_{1/n,\delta}\rfloor}\min\left\{\log(2n),\,(r^*_{\lfloor n/B_{1/n,\delta}\rfloor})^{-1/2}\right\} \le (r^*_{\lfloor n/B_{1/n,\delta}\rfloor})^{1/2} \to 0,
\]
as $n\to\infty$. Therefore, as $n\to\infty$,
\[
\sup_{\alpha\in[\alpha_n, 1]}\sup_{P\in\mathcal{P}}\left(\mathbb{P}_P(\theta(P)\notin \widehat{\mathrm{CI}}_{\alpha,n}) - \alpha\right)_+ \le \exp\left(\sup_{\alpha\in[\alpha_n,1]}B_{\alpha,\gamma/2}r_{\lfloor n/B_{\alpha,\gamma/2}\rfloor}\right) - 1 \to 0.
\]
Hence, $\{\widehat{\mathrm{CI}}_{\alpha,n}:\alpha\in[0, 1], n\ge1\}$ is an asymptotically uniformly valid confidence interval for the set of distributions $\mathcal{P}$. This completes the proof of~\ref{eq:valid-at-gamma} implies~\ref{eq:valid-always}.

\paragraph{\ref{eq:valid-always} $\Rightarrow$~\ref{eq:median-regular}:} Suppose $\{\widetilde{\mathrm{CI}}_{\gamma,n}, \gamma\in[0,1], n\ge1\}$ be an asymptotically uniformly valid confidence interval procedure, i.e., there exists $\alpha_n\to0$ as $n\to\infty$ such that 
\begin{equation}\label{eq:uniform-validity-implication}
\limsup_{n\to\infty} \sup_{\alpha\in[\alpha_n, 1]}\sup_{P\in\mathcal{P}} \left(\mathbb{P}_P(\theta(P)\notin\widetilde{\mathrm{CI}}_{\gamma,n}) - \gamma\right)_+ = 0.
\end{equation}
Define
\[
r_n := \sup_{\alpha\in[\alpha_n,1]}\sup_{P\in\mathcal{P}}\left(\mathbb{P}_P(\theta(P)\notin\widetilde{\mathrm{CI}}_{\gamma,n}) - \gamma\right)_+.
\]
From~\eqref{eq:uniform-validity-implication}, $r_n\to0$ as $n\to\infty$.
Define a new confidence interval procedure $\{\widehat{\mathrm{CI}}_{\alpha,n}:\,\alpha\in[0, 1], n\ge1\}$ as
\[
\widehat{\mathrm{CI}}_{\alpha,n} := \bigcup_{\kappa \ge \alpha}\widetilde{\mathrm{CI}}_{\kappa,n}.
\]
This monotonizes the confidence intervals in terms of the confidence level. Ideally, in a family of confidence intervals, we expect that the confidence interval of confidence $90\%$ (say) contains the confidence interval of confidence $80\%$ (say). The initial given procedure $\widetilde{\mathrm{CI}}_{\gamma, n}, \gamma\in[0, 1]$ may not satisfy this. The defined $\widehat{\mathrm{CI}}_{\gamma,n}, \gamma\in[0, 1]$ does satisfy this. Furthermore,
\begin{align*}
\mathbb{P}(\theta(P)\notin \widehat{\mathrm{CI}}_{\gamma,n}) ~&\le~ \mathbb{P}(\theta(P)\notin \widehat{\mathrm{CI}}_{\kappa,n}) ~\le~ \mathbb{P}(\theta(P)\notin \widetilde{\mathrm{CI}}_{\kappa,n})\quad\mbox{for all}\quad \kappa \ge \gamma.
\end{align*}
Now consider the coverage of $\widehat{\mathrm{CI}}_{n^{-1/2},n}$ as $n\to\infty$. From the monotonicity of $\widehat{\mathrm{CI}}_{\gamma,n}$ in $\gamma\in[0, 1]$, it follows that
\begin{equation}\label{eq:monotonicity-implication}
\mathbb{P}(\theta(P)\notin\widehat{\mathrm{CI}}_{n^{-1/2},n}) ~\le~ 
\begin{cases}
\mathbb{P}(\theta(P)\notin \widetilde{\mathrm{CI}}_{\alpha_n,n}),&\mbox{if }\alpha_n \ge n^{-1/2},\\
\mathbb{P}(\theta(P)\notin \widetilde{\mathrm{CI}}_{n^{1/2},n}),&\mbox{if }\alpha_n \le n^{-1/2}.
\end{cases}
\end{equation}
Recall here that $\alpha_n$ is an unknown quantity from~\eqref{eq:uniform-validity-implication}. Inequality~\eqref{eq:monotonicity-implication} follows from considering the cases $\alpha_n \le n^{-1/2}$ and $\alpha_n \ge n^{-1/2}$. From~\eqref{eq:uniform-validity-implication}, we obtain
\[
\mathbb{P}(\theta(P)\notin\widetilde{\mathrm{CI}}_{\alpha_n,n}) \le \alpha_n + r_n,
\]
and if $\alpha_n \le n^{-1/2} \le 1$, then
\[
\mathbb{P}(\theta(P)\notin\widetilde{\mathrm{CI}}_{n^{-1/2},n}) \le n^{-1/2} + r_n.
\]
Hence, for all $n\ge1$,
\[
\sup_{P\in\mathcal{P}}\mathbb{P}(\theta(P)\notin \widehat{\mathrm{CI}}_{n^{-1/2},n}) \le \max\{\alpha_n, n^{-1/2}\} + r_n.
\]
Lemma~\ref{lem:confidence-median-bias} now implies that there exists an estimator $\widehat{\theta}_n$ obtained from the confidence interval $\widehat{\mathrm{CI}}_{n^{-1/2},n}$ such that
\[
\sup_{P\in\mathcal{P}}\mbox{Med-bias}_{\theta(P)}(\widehat{\theta}_n; P) \le \frac{\max\{\alpha_n, n^{-1/2}\}}{2}.
\]
Because $\max\{\alpha_n, n^{-1/2}\}$ converges to zero as $n\to\infty$, we obtain an estimator sequence $\{\widehat{\theta}_n, n\ge 1\}$ such that
\[
\limsup_{n\to\infty}\sup_{P\in\mathcal{P}}\mbox{Med-bias}_{\theta(P)}(\widehat{\theta}_n; P) = 0.
\]
This completes the proof of~\ref{eq:valid-always} $\Rightarrow$~\ref{eq:median-regular}.
\paragraph{\ref{eq:median-regular} $\Rightarrow$~\ref{eq:valid-at-gamma}:} Suppose $\widehat{\theta}_n, n\ge 1$ be an estimator sequence such that
\[
\limsup_{n\to\infty}\,\sup_{P\in\mathcal{P}}\mbox{Med-bias}_{\theta(P)}(\widehat{\theta}_n; P) = 0.
\]
Define
\[
s_n = \sup_{P\in\mathcal{P}}\mbox{Med-bias}_{\theta(P)}(\widehat{\theta}_n; P).
\]
By the hypothesis, $s_n\to0$ as $n\to\infty$. Lemma~\ref{lem:median-bias-confidence} implies that there exists a confidence interval procedure $\{\widehat{\mathrm{CI}}_{\alpha,n}: \alpha\in[0, 1], n\ge1\}$ such that for all $\alpha\in[0, 1]$,
\[
\sup_{P\in\mathcal{P}}\mathbb{P}(\theta(P)\notin\widehat{\mathrm{CI}}_{\alpha,n}) \le \alpha\left(1 + 2s_{\lfloor n/B_{\alpha,0}\rfloor}\right)^{B_{\alpha,0}}.
\]
Therefore,
\[
\sup_{P\in\mathcal{P}}\mathbb{P}(\theta(P)\notin\widehat{\mathrm{CI}}_{\gamma,n})\le \gamma\left(1 + 2s_{\lfloor n/B_{\gamma,0}\rfloor}\right)^{B_{\gamma,0}}.
\]
With $\gamma\in[0, 1]$ being a fixed number and $B_{\gamma, 0} = \lceil \log_2(2/\gamma)\rceil$, $s_{\lfloor n/B_{\gamma,0}\rfloor} \to 0$ as $n\to\infty$. Hence it follows that
\[
\limsup_{n\to\infty}\sup_{P\in\mathcal{P}}\mathbb{P}(\theta(P)\notin\widehat{\mathrm{CI}}_{\gamma,n})\le \gamma.
\]
This completes the proof of~\ref{eq:median-regular} $\Rightarrow$~\ref{eq:valid-at-gamma}.
\section{Summary and Conclusions}
We introduced a new notion of regularity of an estimator and proved that uniformly valid inference for a functional is possible if and only if there exists a median regular estimator for that functional. The method of proof is constructive, explicitly providing an estimator from a uniformly valid confidence interval and vice versa. The confidence interval obtained from a median regular estimator is constructed via the recently proposed HulC method~\citep{kuchibhotla2021hulc}. As a simple implication of the results of this paper we can conclude that inference for a functional is possible if and only if the HulC methodology yields valid inference with some estimator. Such an implication, to our knowledge, is not true for any other inference procedure including the bootstrap, and subsampling.  

Although the method of proof is constructive, the result is mainly of theoretical interest currently. In many simple examples, we know that the uniformly valid confidence intervals or the median regular estimators obtained via the proof are sub-optimal and it is an important research direction to formalize both a useful notion of ``optimality'' as well as to study ``optimal'' constructions, i.e. procedures which yield optimal median regular estimators from (optimal) confidence intervals and vice versa. 
\bibliographystyle{apalike}
\bibliography{references}

\begin{thebibliography}{}

\bibitem[Andrews and Guggenberger, 2010]{andrews2010asymptotic}
Andrews, D.~W. and Guggenberger, P. (2010).
\newblock Asymptotic size and a problem with subsampling and with the m out of
  n bootstrap.
\newblock {\em Econometric Theory}, pages 426--468.

\bibitem[Bahadur and Savage, 1956]{bahadur1956nonexistence}
Bahadur, R.~R. and Savage, L.~J. (1956).
\newblock The nonexistence of certain statistical procedures in nonparametric
  problems.
\newblock {\em The Annals of Mathematical Statistics}, 27(4):1115--1122.

\bibitem[Desu and Rodine, 1969]{mahamunulu1969estimation}
Desu, M.~M. and Rodine, R.~H. (1969).
\newblock Estimation of the population median.
\newblock {\em Scandinavian Actuarial Journal}, 1969(1-2):67--70.

\bibitem[Donoho, 1988]{donoho1988one}
Donoho, D.~L. (1988).
\newblock One-sided inference about functionals of a density.
\newblock {\em The Annals of Statistics}, pages 1390--1420.

\bibitem[Fan and Park, 2010]{Fan2010}
Fan, Y. and Park, S.~S. (2010).
\newblock Sharp bounds on the distribution of treatment effects and their
  statistical inference.
\newblock {\em Econometric Theory}, 26(3):931--951.

\bibitem[Hirano and Porter, 2012]{hirano2012impossibility}
Hirano, K. and Porter, J.~R. (2012).
\newblock Impossibility results for nondifferentiable functionals.
\newblock {\em Econometrica}, 80(4):1769--1790.

\bibitem[Imbens and Manski, 2004]{imbens2004confidence}
Imbens, G.~W. and Manski, C.~F. (2004).
\newblock Confidence intervals for partially identified parameters.
\newblock {\em Econometrica}, 72(6):1845--1857.

\bibitem[Kuchibhotla et~al., 2021]{kuchibhotla2021hulc}
Kuchibhotla, A.~K., Balakrishnan, S., and Wasserman, L. (2021).
\newblock The hulc: Confidence regions from convex hulls.
\newblock {\em arXiv preprint arXiv:2105.14577}.

\bibitem[Leeb and P{\"o}tscher, 2005]{leeb2005model}
Leeb, H. and P{\"o}tscher, B.~M. (2005).
\newblock Model selection and inference: Facts and fiction.
\newblock {\em Econometric Theory}, 21(1):21--59.

\bibitem[Li, 1989]{li1989honest}
Li, K.-C. (1989).
\newblock Honest confidence regions for nonparametric regression.
\newblock {\em The Annals of Statistics}, 17(3):1001--1008.

\bibitem[Loh, 1984]{loh1984estimating}
Loh, W.-Y. (1984).
\newblock Estimating an endpoint of a distribution with resampling methods.
\newblock {\em The Annals of Statistics}, 12(4):1543--1550.

\bibitem[Michel, 1974]{michel1974bounds}
Michel, R. (1974).
\newblock Bounds for the efficiency of approximately median unbiased estimates.
\newblock {\em Metrika}, 21(1):205--211.

\bibitem[Pfanzagl, 1970]{pfanzagl1970asymptotic}
Pfanzagl, J. (1970).
\newblock On the asymptotic efficiency of median unbiased estimates.
\newblock {\em Annals of Mathematical Statistics}, 41(5):1500--1509.

\bibitem[Pfanzagl, 1979]{pfanzagl1979optimal}
Pfanzagl, J. (1979).
\newblock On optimal median unbiased estimators in the presence of nuisance
  parameters.
\newblock {\em Annals of Statistics}, 7(1):187--193.

\bibitem[Pfanzagl, 1998]{pfanzagl1998nonexistence}
Pfanzagl, J. (1998).
\newblock The nonexistence of confidence sets for discontinuous functionals.
\newblock {\em Journal of statistical planning and inference}, 75(1):9--20.

\bibitem[Pfanzagl, 2000]{pfanzagl2000regularly}
Pfanzagl, J. (2000).
\newblock Are regularly estimable functionals always differentiable?
\newblock {\em aequationes mathematicae}, 59(3):262--272.

\bibitem[Pfanzagl, 2011]{pfanzagl2011parametric}
Pfanzagl, J. (2011).
\newblock {\em Parametric statistical theory}.
\newblock Walter de Gruyter.

\bibitem[P{\"o}tscher, 2002]{potscher2002lower}
P{\"o}tscher, B.~M. (2002).
\newblock Lower risk bounds and properties of confidence sets for ill-posed
  estimation problems with applications to spectral density and persistence
  estimation, unit roots, and estimation of long memory parameters.
\newblock {\em Econometrica}, 70(3):1035--1065.

\bibitem[Robins, 2004]{robins2004optimal}
Robins, J.~M. (2004).
\newblock Optimal structural nested models for optimal sequential decisions.
\newblock In {\em Proceedings of the second seattle Symposium in
  Biostatistics}, pages 189--326. Springer.

\bibitem[Romano and Shaikh, 2012]{romano2012uniform}
Romano, J.~P. and Shaikh, A.~M. (2012).
\newblock On the uniform asymptotic validity of subsampling and the bootstrap.
\newblock {\em The Annals of Statistics}, 40(6):2798--2822.

\bibitem[Van Der~Vaart, 1991]{van1991differentiable}
Van Der~Vaart, A. (1991).
\newblock On differentiable functionals.
\newblock {\em The Annals of Statistics}, pages 178--204.

\end{thebibliography}
\end{document}